\documentclass[11pt, reqno]{article}

\usepackage{amsmath, amssymb, amsthm}
\usepackage[margin=1in]{geometry}
\usepackage{verbatim, dsfont}
\usepackage{graphicx}
\usepackage[shortlabels]{enumitem}
\usepackage{listings}
\usepackage[ruled,vlined,linesnumbered]{algorithm2e}

\newcommand{\angles}[1]{\left\langle #1 \right\rangle}
\newcommand{\bb}{\mathbb}

\newcommand{\E}{\bb E}
\newcommand{\F}{\bb F}

\newcommand{\ind}{\mathds{1}}
\newcommand{\inv}{^{-1}}

\newcommand{\mc}{\mathcal}
\newcommand{\mf}{\mathfrak}
\newcommand{\norm}[1]{\left\Vert#1\right\Vert}

\newcommand{\R}{\bb R}

\newcommand{\Z}{\bb Z}
\renewcommand{\P}{\bb P}

\newtheorem{theorem}{Theorem}[section]

\newtheorem{corollary}[theorem]{Corollary}

\newtheorem{fact}[theorem]{Fact}
\newtheorem{lemma}[theorem]{Lemma}

\newtheorem{proposition}[theorem]{Proposition}

\usepackage{xcolor}

\title{Popular Differences for Corners in Abelian Groups}
\author{Aaron Berger\thanks{Department of Mathematics, MIT \textit{bergera@mit.edu}}}
\date{}

\begin{document}

\maketitle

\begin{abstract}
	For a compact abelian group $G$, a \textit{corner} in $G \times G$ is a triple of points $(x,y)$, $(x,y+d)$, $(x+d,y)$. The classical {corners theorem} of Ajtai and Szemer\'edi implies that for every $\alpha > 0$, there is some $\delta > 0$ such that every subset $A \subset G \times G$ of density $\alpha$ contains a $\delta$ fraction of all corners in $G \times G$, as $x,y,d$ range over $G$. 
	
	Recently, Mandache proved a ``popular differences'' version of this result in the finite field case $G = \F_p^n$, showing that for any subset $A \subset G \times G$ of density $\alpha$, one can fix $d \neq 0$ such that $A$ contains a large fraction, now known to be approximately $\alpha^4$, of all corners with difference $d$, as $x,y$ vary over $G$. We generalize Mandache's result to all compact abelian groups $G$, as well as the case of corners in $\Z^2$.
\end{abstract}

\section{Introduction}
The following \textit{popular differences} version of Szemer\'edi's theorem was conjectured by Bergelson, Host, and Kra \cite{bhk05} and proved by Green \cite{g05} for $k = 3$  and Green-Tao \cite{gt10} for $k = 4$: every subset of $[N]$ of size at least $\alpha N$ contains at least $(\alpha^k - o(1))N$ $k$-term arithmetic progressions, or $k$-APs, with the same common difference. That is, such a set contains $(\alpha^k - o(1))N$ distinct copies of
$
\{x, x+d, \ldots,x+(k-1)d\}
$
for some fixed $d \neq 0$. These results involve the method of arithmetic regularity developed by Green, and the lower bounds are essentially best possible; a randomized construction gives subsets of density $\alpha$ and only $(\alpha^k + o(1))N$ $k$-APs
with common difference $d$ for each $d \neq 0$. Such polynomial bounds for AP counts are not the norm in additive combinatorics. Indeed, in an appendix to \cite{bhk05}, Rusza shows that for $k \ge 5$, one can construct sets with density $\alpha$ and fewer than $\alpha^{o(1)}$ distinct $k$-APs\footnote{The $o(1)$ term goes to 0 as $\alpha \to 0$ and $N \to \infty$.}  with common difference $d$ for each $d \neq 0$. The natural place to look for generalizations is in higher-dimensional configurations. The \textit{corners theorem} of Ajtai and  Szemer\'edi \cite{as74} is a classical result in this style in two dimensions, implying that any subset of $[N]^2$ with at least $\alpha N^2$ elements contains $\Omega(N^3)$ corners, which are triples of the form
$
\{(x,y), (x,y+d), (x+d,y)\}.
$
As usual, the dependence of the implicit constant in $\Omega(N^3)$ on the density $\alpha$ is quite poor. One might hope to obtain a better dependence for some fixed $d$ than what one obtains on average by the Ajtai-Szemer\'edi result. The following result due to Mandache \cite{m18} does precisely this, but in the finite field model instead of $[N]$. For a family $\mc F$ of finite abelian groups, let $M_{\mc F}(\alpha) \in [0,1]$ be the minimum value such that the following statement is true: For every $A \subset G \times G$ with size at least $\alpha |G|^2$, there is some $d\neq 0$ such that $A$ contains $(M_{\mc F}(\alpha) - o(1))|G|^2$ corners with common difference $d$, where the $o(1)$ term goes to 0 as $|G| \to \infty$. Mandache shows that for fixed $p$ and $\mc F = \{\F_p^n\}$, one has
$$
m'(\alpha) \le M_{\mc F}(\alpha) \le m(\alpha),
$$ 
where $m'(\alpha)$ and $m(\alpha)$ are polynomially large in terms of $\alpha$ and are given by the solutions to a certain variational problem we describe below. In a somewhat surprising difference from the $k$-AP case, Mandache shows that the exponents in the growth rates of $m(\alpha), m'(\alpha)$ are strictly greater than 3, whereas random subsets of $G \times G$ have approximately $\alpha^3|G|^2$ corners for each fixed difference $d \neq 0$. The asymptotic growth rates of $m$ and $m'$ were recently determined by Fox, Sah, Sawhney, Stoner, and Zhao \cite{fsssz19}, who also discuss other possible generalizations and barriers to generalization for popular differences results. We will include their bounds on $m$ and $m'$ following the discussion of the variational problem itself.

For $\phi:[0,1]^3 \to [0,1]$, define
\begin{equation*}
T(\phi) := \int_{[0,1]^3}dx\,dy\,dz \int_{[0,1]}\phi(x,y,z')\,dz'\int_{[0,1]}\phi(x,y',z)\,dy'\int_{[0,1]}\phi(x',y,z)\,dx'.
\end{equation*}
We are concerned with the infimum of $T(\phi)$ over $\phi$ with a fixed expectation:
$$
m(\alpha) := \inf_{\substack{\phi:[0,1]^3 \to [0,1]\\\E[\phi] = \alpha}} T(\phi).
$$
This expression may be rewritten by taking independent $X,Y,Z \sim $ Unif$([0,1])$, in which case one has
\begin{equation}\label{eqn:variational problem}
T(\phi) = \E\big[\E(\phi|X,Y)\E(\phi|X,Z)\E(\phi|Y,Z)\big].
\end{equation}
It is clear that the underlying probability space is unimportant here; if $X,Y,Z$ are any independent random variables and $\phi$ has expectation $\alpha$, then $T(\phi) \ge m(\alpha)$. 

Mandache showed that for any family of finite abelian groups, one has
\begin{equation*}
M(\alpha) \le m(\alpha).
\end{equation*}
Secondly, let $m'(\alpha)$ be the maximal convex function such that $m'(\alpha ) \le m(\alpha)$ pointwise. Mandache proved that for fixed $p$ and $\mc F = \{F_p^n\}$, one has
\begin{equation*}
M_{\mc F}(\alpha) \ge m'(\alpha).
\end{equation*}
More specifically, for the lower bound Mandache showed that for $A \subseteq G \times G$ with density $\mu(A) = \alpha$, there is a subspace $W \subseteq G$ with codimension bounded in terms of $\epsilon$ so that 
$$
\E_{x,y \in G, d \in W} \big[\ind_A(x,y) \ind_A(x,y+d)\ind_A(x+d,y)\big] \ge m'(\alpha)-\epsilon.
$$
Letting $n \to \infty$, by the boundedness of codim$(W)$, the corners with difference $d = 0$ contribute $o(1)$ to this expectation, and so he concludes that there is some $d \neq 0$ with 
$$
\E_{x,y \in G} \big[\ind_A(x,y) \ind_A(x,y+d)\ind_A(x+d,y)\big] \ge m'(\alpha)-O(\epsilon).
$$
Since this inequality holds for every $\epsilon$ as $n \to \infty$, we obtain the popular differences result $M_{\mc F}(\alpha) \ge m'(\alpha)$. Mandache showed that
$$
\alpha^4 \le m'(\alpha) \le m(\alpha) \le C\alpha^{3.13}.
$$
Fox, Sah, Sawhney, Stoner, and Zhao \cite{fsssz19} determined more precise asymptotics, showing:
$$
\omega(\alpha^4) \le m'(\alpha) \le m(\alpha) \le \alpha^{4-o(1)},
$$
where the $o(1)$ term approaches 0 as $\alpha \to 0$, and the $\omega(\alpha^4)$ term is $\alpha^4/o(1)$.

We generalize Mandache's result to all compact abelian groups.
\begin{theorem}\label{thm:main thm}
	For any $\alpha, \epsilon > 0$, there is some absolute $c > 0$ such that the following holds: For any compact abelian group $G$ with Haar probability measure $\mu$ and any set $A \subseteq G \times G$ with $\mu(A) = \alpha$, there is a Bohr set $B \subseteq G$ with $\mu(B) \ge c$ such that 
	\begin{equation*}
	\int_{\substack{x,y \in G\\r \in B}} \ind_A(x,y) \ind_A(x,y+r) \ind_A(x+r,y)~dx\,dy\,dr \ge m'(\alpha) - \epsilon.
	\end{equation*}
	
\end{theorem}
From this result and a simple modification we obtain the following two corollaries.
\begin{corollary}
	Let $G$ be any finite abelian group and $A \subseteq G \times G$ have size $|A| \ge \alpha |G|^2$. Then there is some $r \neq 0$ such that $A$ contains at least $(m'(\alpha) - o(1))|G|^2$ corners of the form $\{(x,y),(x,y+r),(x+r,y)\}$.
\end{corollary}
\begin{corollary}\label{thm:corners in Z}
	Let $A \subset [n]^2$ have size $|A| \ge \alpha n^2$. Then there is some $r \neq 0$ such that $A$ contains at least $(m'(\alpha) - o(1))|G|^2$ corners of the form $\{(x,y),(x,y+r),(x+r,y)\}$.
\end{corollary}

\subsection{Notation}
Let $(G,+)$ be a compact abelian group, with Haar probability measure $\mu$, and a (discrete) dual group $\hat G$ of characters $\xi: G \to \R/\Z$. We will use function evaluation notation for characters, so $\xi(x)$ denotes the image of $x \in G$ under $\xi \in \hat G$.
For a measurable function $f: G \to \R$ and a measurable partition $P$ of $G$, we let $f_{P} = \E(f|P)$ be the function obtained by averaging $f$ on each part of $P$. For measurable $X \subseteq G$ with $\mu(X) > 0$, define 
$$
\mu_X := \frac{\ind_X}{\mu(X)},
$$
to be the indicator of $X$, normalized to have integral 1.\\

For asymptotics, we use $x = O(y)$ and $x \lesssim y$ when we would otherwise write $x \le Cy$ for some absolute constant $C$. An absolute constant is independent of any variables in the problem. For example, it suffices to prove Theorem \ref{thm:main thm} with $m'(\alpha) - \epsilon$ replaced by $m'(\alpha) - O(\epsilon)$, as the implicit constant is independent of $\epsilon$.

For $f : G \to \R$, we use the $L^p$ norms, normalized as follows.
$$
\norm{f}_{L^p} = \left(\int_{G} |f(x)|^p~dx\right)^{1/p}.
$$
For $\hat f: \hat G \to \R/\Z$, we use $\ell^p$ norms.
$$
\big\Vert\hat f\big\Vert_{\ell^p} = \left(\sum_{\xi \in \hat G} (\hat f(\xi))^p\right)^{1/p}.
$$
Similarly, Fourier transforms are written with an integral over the real domain and a sum over the frequency domain, so $\hat f(\xi) = \int_G f(x)e^{-2\pi i\xi(x)} dx$ and $f(x) = \sum_{\xi \in \hat G} \hat f(\xi)e^{2\pi i\xi(x)}$. Using this notation, Plancherel's theorem states $\norm{f}_{L^2} = \big\Vert\hat f\big\Vert_{\ell^2}$.
Finally, for $x \in \R$ or $\R/\Z$, we write $\norm{x}_{\R/\Z}$ to mean the distance from $x$ to the nearest integer.

\section{Bohr set preliminaries}
The \textit{Bohr set} given by a finite set of frequencies $S \subset \hat G$ and $\rho > 0$ is defined to be
$$
B(S,\rho) = \{x \in G: \sup_{\xi \in S}\norm{\xi (x)}_{\R/\Z} < \rho\}.
$$
For $\delta = 1/N$, we also define the Bohr partition $\mf B(S,\delta)$ to be the union of parts of the form
$$
\left\{x \in G : \xi_i \cdot x \in \left[\frac{s_i - 1}{N}, \frac{s_i}{N}\right) ~~\forall i \in [d]\right\}, 
$$ 
for each choice of $\{s_i\} \in [N]^d$. The number of parts in a Bohr partition is $|\mf B(S,\delta)| = \delta^{-|S|}$.
Each Bohr set has size bounded below by a constant depending on $\rho$ and $|S|$:
\begin{fact}\label{fact:volume bound}
	For any Bohr set $B(S,\rho)$, there exists a constant $C_{|S|,\rho} > 0$ depending only on $\rho$ and $|S|$ such that:
	$$
	\mu(B(S,\rho)) \ge C_{|S|,\rho}.
 	$$
\end{fact}
\begin{proof}
	Consider the maximal $\delta < \rho, \delta = 1/N$. By the triangle inequality, for any $x \in G$, whichever part of $\mf B(S,\delta)$ contains $x$ is itself entirely contained in $x+B(S,\rho)$. Choosing one representative $x$ from each nonempty part of $\mf B(S, \delta)$, we see that $N^{|S|}$ translates of $B(S,\rho)$ suffice to cover $G$.
\end{proof}
When drawing analogies between the finite field model and the case of general abelian groups, Bohr sets take the role of subspaces. One major problem with the general setting is that Bohr sets, unlike subspaces, are not closed under addition. The common approach to handle this relies on the fact that Bohr sets are approximately closed under \textit{addition by elements of much smaller Bohr sets}. The properties we need are collected in Proposition \ref{prop:sets with partitions} and Corollary \ref{cor:set convolution}, and may be obtained without relying on the regular neighborhoods of Bourgain or the smoothed neighborhoods of Tao (for reference, see \cite{b99, g05,t14}).

In this proposition we look at the interplay between a ``coarse'' partition $\mf B(S, \delta)$, a ``fine'' partition $\mf B(S', \delta')$, and an ``intermediate'' Bohr set $B(S,\rho)$. As long as $\rho$ is a sufficiently small with respect to $\delta$ and a sufficiently large with respect to $\delta'$, we have that almost all translates of $B(S, \rho)$ lie inside a single part of $\mf B(S, \delta)$, and almost all parts of $\mf B(S',\delta')$ that intersect a fixed translate of $B(S,\rho)$ are entirely contained in that translate.
\begin{proposition}\label{prop:sets with partitions}
	Let $S \subseteq S' \subseteq \hat G$ and fix $\epsilon_0 > 0$. We have:
	\begin{enumerate}
		\item If $\rho \le \epsilon_0 \delta/|S|$, then for all but an $O(\epsilon_0)$-fraction of $x \in G$, the Bohr set translate $x + B(S, \rho)$ is entirely contained in a single part of the Bohr partition $\mf B(S, \delta)$.
		\item If $\delta' \le \epsilon_0 C_{|S|, \rho}/|S|$,
		\footnote{This is the $C_{|S|,\rho}$ from Fact \ref{fact:volume bound}--we are simply requiring $\delta'$ to be smaller than some constant depending on $|S|, \rho, \epsilon_0$.}
		then for all $x \in G$ and all but an $O(\epsilon_0)$-fraction of $y \in B(S, \rho)$, $x+y$ lies in a part of $\mf B(S', \delta')$ that is entirely contained in $x + B(S,\rho)$. 
	\end{enumerate}
\end{proposition}
\begin{proof}
	The strategy is to show that the image of elements of $G$ under a character $\xi$ are either evenly distributed in $\R/\Z$ or do not affect our computation. For those which are evenly distributed, a simple union bound suffices to show that most $x \in G$ are not close to the boundary of a Bohr set or Bohr part in the ``direction of'' any character. 
	
	We begin with the proof of Part 1. To determine which part of a Bohr partition contains $x \in G$, it suffices to determine the values of $\xi (x)$ for each $\xi \in S$. For $\xi \in S$, we consider two possibilities. If there is no $x_0 \in G$ with $0 < \norm{\xi (x_0)}_{\R/\Z} < \rho$, then adding any element of $B(S, \rho)$ to any $x \in G$ will not change the value of $\xi(x)$, and so we may ignore such $\xi$.
	
	Otherwise, there exists $x_0 \in G$ with $0 < \norm{\xi(x_0)}_{\R/\Z} < \rho$. In this case, since the map $x \mapsto x + x_0$ is measure-preserving, the sets $$\{x \in G: \xi(x) \in [(k-1)\xi(x_0), k \xi(x_0)]\}$$ are of equal measure. A union of $\lceil 1/(\xi(x_0)) \rceil = \Theta(1/\xi(x_0))$ of these sets cover $G$, and so each interval has measure $\Theta(1/(\xi(x_0)))$. By translation, for any interval $I \subset \R/\Z$ with length $|I| \ge \xi(x_0)$, the set $\{x \in G: \xi(x) \in I\}$ has measure $\Theta(|I|)$. Thus, the set 
			$$S_{\delta, \xi}:=\{x \in G: \norm{\xi(x) - k\delta }_{\R/\Z} \le \rho \text{ for some } k \in \Z\}$$
			is a union of $O(1/\delta)$ preimages under $\xi$ of intervals of measure $2\rho \ge \xi(x_0)$, and so it has measure 
			$$\mu(S_{\delta, \xi}) \lesssim \rho/\delta \lesssim \epsilon_0/|S|.$$

	For any $x \notin S_{\delta, \xi}$, by triangle inequality, adding any $y \in B$ cannot change the value of the largest multiple of $\delta$ less than $\xi(x)$, and summing this up over all $\xi \in S$ gives a subset of measure $O(\epsilon_0)$ which contains all the elements of $x$ that are bad for some $\xi$, which completes Part 1.\\
	
	Part 2 proceeds in a similar manner. For any $y \in G$ lying in some part of $p \in \mf B'$ we know that $p \subseteq y+B(S', \delta') \subseteq y+B(S,\delta')$. It therefore suffices to show that for all $x$ and all but an $\epsilon_0$-fraction of $y \in x+B(S, \rho)$, we have $y + B(S, \delta') \subset x+B(S,\rho)$. By translation we may assume $x = 0$. 
	
	Let $\xi \in S$.
	If there is no $x_0 \in G$ with $0 < \norm{\xi(x_0)}_{\R/\Z} < \rho'$, then adding any element of $B(S', \rho')$ to any $x \in G$ will not change the value of $\xi(x)$, and so we may ignore such $\xi$.
	
	Otherwise, there exists $x_0 \in G$ with $0 < \norm{\xi(x_0)}_{\R/\Z} < \rho'$. In this case, since the map $x \mapsto x + x_0$ is measure-preserving, the sets $$\{x \in G: \xi(x) \in [(k-1)\xi(x_0), k \xi(x)_o)\}$$ are of equal measure, and so the exceptional set 
		$$E_\xi := \{x \in G: \norm{\xi(x) - \rho}_{\R/\Z} \le \rho' \} \cup \{x \in G: \norm{\xi(x) + \rho}_{\R/\Z} \le \rho' \}$$
		has measure bounded by $O(\rho') \lesssim \epsilon/|S|C_{|S|, \rho}$. 
	
	By the triangle inequality, we have $x+B(S,\rho') \subseteq B(S,\rho)$ as long as $x$ is not contained in $\bigcup_{\xi \in S} E_\xi$. A simple union bound tells us that this set has size $\lesssim \epsilon_0/|B(S,\rho)|$, as desired.
\end{proof}
Reproducing the argument for the second half of this theorem when $B(S,\rho)$ is replaced by an arbitrary translate of an arbitrary part $p \in \mf B(S,\rho)$, we can nearly obtain the same conclusion. However, Bohr parts may have wildly varying size. If we replace $C_{|S|,\rho}$ by $\epsilon_0/|\mf B(S,\rho)| = \epsilon_0/\rho^{-|S|}$, then the conclusion will hold for all Bohr parts with size at least an $\epsilon_0$-fraction of the average size of a Bohr part, which will be plenty. 

This proposition and observation allow us to make the following useful decompositions. 
\begin{corollary}\label{cor:box approx}
	Fix $\epsilon_0, |S|, \rho$, and let $B$ be either the Bohr set $B(S,\rho)$ or a part of the Bohr partition $\mf B(S, \rho)$ with size $\mu(B) \ge \epsilon_0/|\mf B(S,\rho)|$. There exists some $C$ depending only on $\epsilon_0, |S|, \rho$ such that for any $z_0 \in G$, the set 
	$$
	B_{z_0} := \{(x,y) \in G\times G: x+y+z_0 \in B\}
	$$
	can be expressed as the disjoint union of at most $C$ boxes and a remainder of measure at most $\epsilon_0 \mu(B)$.
\end{corollary}
\begin{proof}
	Let us temporarily fix $y$. We will use a fine partition $\mf B' = \mf B(S, \rho')$; the boxes of $\mf B' \times \mf B'$ should mostly cover our set. By Proposition \ref{prop:sets with partitions} and the subsequent comments, as long as $\mf \rho'$ is sufficiently small in terms of $\epsilon_0$, $|S|$, and $\rho$, for all but an $\epsilon_0/2$-fraction of $x \in y+z_0+B(S,\rho')$, we have that the part of $\mf B'$ containing $x$ lies entirely within $y+z_0+B$. Varying $y$, this statement holds for the $x$-coordinate of all but an $\epsilon_0$-fraction of pairs $(x,y)$ in $B_{z_0}$. We can repeat the same argument for the $y$-coordinate. Combining these together, for all but am $\epsilon_0$-fraction of $(x,y)$ in $B_{z_0}$, we have that the box of $\mf B' \times \mf B'$ which contains $(x,y)$ is itself fully contained within $B_{z_0}$. Consequently, $B_{z_0}$ may be partitioned into a union of at most $C := |\mf B'|^2$ boxes and an exceptional set of measure at most $\epsilon_0$. Since $|\mf B'|$ is bounded in terms of $\epsilon_0$, $|S|$, and $\rho$, this completes the proof of the corollary.
\end{proof}

Integrating the pointwise statements of Proposition \ref{prop:sets with partitions}, we can obtain a second useful corollary.
\begin{corollary}\label{cor:set convolution}
	Fix $\epsilon_0, \delta,\delta',\rho > 0$ and $S \subseteq S' \subseteq \hat G$. We let $\mf B = \mf B(S,\delta)$, $\mf B' = \mf B(S', \delta')$, and $B = B(S,\rho)$. Furthermore, assume 
	$$\rho \le \frac{\epsilon_0^2\delta}{|S|} \quad \text{and} \quad
	\delta' \le \frac{\epsilon_0C_{|S|,\rho}}{|S|}.$$  Then for any $f:G \to[0,1]$ we have:  
	
		\begin{equation}\label{eqn:set conv eqn 1}
		\big\Vert f|_{\mf B} - \mu_{B} * f|_{\mf B} \big\Vert_{L^2} \lesssim \epsilon_0.
		\end{equation}

		\begin{equation}\label{eqn:set conv eqn 2}
		\big\Vert\mu_B * f - \mu_B * f|_{\mf B'}\big\Vert_{L^2} \lesssim \epsilon_0.
		\end{equation}
\end{corollary}
Intuitively, (\ref{eqn:set conv eqn 1}) says that a function which is constant on a coarse Bohr partition is approximately constant under convolution with a small Bohr set, and (\ref{eqn:set conv eqn 2}) states that convolving a function with a Bohr set is approximately the same as first projecting onto a much finer Bohr partition, and then performing the convolution.
\begin{proof}
		By the first half of Proposition \ref{prop:sets with partitions}, the set 
		$$
		\{x: f|_{\mf B}(x) \neq f|_{\mf B} * \mu_{B(S,\rho)}(x)\}
		$$
		has measure bounded by $\epsilon_0^2$. 
		As the difference of two functions with range in $[0,1]$, we have 
		$$|f|_{\mf B}(x) - f|_{\mf B} * \mu_{B}(x)|^2 \le 1.$$
		Since this function is nonzero on a set of measure at most $\epsilon_0^2$,  (\ref{eqn:set conv eqn 1}) follows immediately.
		
		To show (\ref{eqn:set conv eqn 2}) we apply the second half of Proposition \ref{prop:sets with partitions}. For any $x \in G$, this allows us to partition $x+B(S,\rho)$ into the union of some Bohr parts $b \in \mf B$ and an exceptional set $E$ with $\mu(E) \le \epsilon_0^2 \mu(B)$. Observing that the integral of $f$ equals the integral of $f|_{\mf B}$ on such a part $b$, we obtain:
		\begin{align*}
		\mu_B * f - \mu_B * f|_{\mf B'}(x) &= \int_{y \in G}(f - f|_{\mf B'})(x+y)~\mu_{B(S,\rho)}(-y) ~dy\\
		&= \frac{1}{\mu(B)}\int_{y' \in x+B(S,\rho)}(f-f|_{\mf B'})(y') ~dy'\\
		&= \frac{1}{\mu(B)}\int_{y' \in E}(f -f|_{\mf B'})(y') ~dy'.
		\end{align*}
		We take absolute values. The integrand has absolute value bounded by 1, and is supported on a set of measure at most $\epsilon_0 \mu(B)$. We deduce:
		\begin{equation*}
		\big\Vert\mu_B * f - \mu_B * f|_{\mf B'}\big\Vert_{L^2} \le \big\Vert\mu_B * f - \mu_B * f|_{\mf B'}\big\Vert_{L^\infty} \le \epsilon_0.
		\end{equation*}
\end{proof}
\section{Regularity lemma}
	We will require two types of regularity lemmas. The first allows us to decompose a function, or a set of functions, into three parts: one that is constant on a Bohr partition, one that is small in $L^1$, and one that is Fourier uniform. The second type of regularity is standard strong regularity for graphs or graphons.
	%TODO need a reference here.
%\subsection{Strong Bohr set regularity}\label{sec:Strong Bohr Reg}
	 
	\begin{lemma}\label{lem:strong bohr regularity}
		Fix $\epsilon$, $m$, and $F:\R\to \R$, a rapidly growing function whose choice may depend on $\epsilon$ and $m$. Then there exist constants $D, R$ such that the following holds. For every set $\mc I$ of functions $I : G \to [0,1]$ with cardinality $|\mc I| = m$, there exists a Bohr partition $\mf B = \mf B(S,\rho)$ with $|S| < D$, $\rho > R$, and a decomposition 
		$$
		I = I_0+I_1+I_2
		$$
		for each $I \in \mc I$, such that:
		\begin{equation*}
			I_0 = I|_{\mf B},\qquad 
			\norm{I_1}_{L^2} \lesssim \frac{1}{F(1)},\quad\text{and}
			\quad \big\Vert\widehat{ I_2\cdot \ind_b}\big\Vert_{\ell^\infty} \lesssim \frac{1}{F(\delta_i\inv |\mf B|)} ~\text{for all}~ b \in \mf B.
		\end{equation*}
	\end{lemma}
	
	The proof of this lemma will occupy the remainder of this subsection. For this lemma we use a procedure in which we will be constructing a sequence of Bohr sets $B(S_i, \rho_i)$. Each Bohr set will be accompanied by a Bohr partition $\mf B(S_i, \delta_i)$ with $\delta_i$ substantially smaller than $\rho_i$, but by a bounded amount. At each successive refinement, we regularize an increasingly large family of functions $\mc F_i$ with respect to the previous Bohr set. The procedure is as follows:
	\begin{enumerate}
		\item Initialize $S_0 = \emptyset$, $\rho_0 = 1$.
		\item Set $P_i = \mf B(S_i, \delta_i)$, where $1/\delta_i \ge F(1/\rho_i)$ is chosen to be an integer and, for $i \ge 1$, a multiple of $1/\delta_{i-1}$.
		\item Set $\mc F_{i}$ to be the set of pointwise products of functions $I \cdot \ind_p$, for all $I \in \mc I$ and $p \in P_i$.
		\item Set
		\begin{align*}
		S_{i+1} = S_i ~\cup~& \{\xi \in \hat G: \widehat f(\xi) \ge 1/F(|\mc F_i|/\delta_i)\text{~ for some } f \in \mc F_i \}.
		%\\~\cup~& \{\xi \in G: \widehat{\nu_{j}}(\xi) \ge \delta_i \text{ for some } j \le i\}. 
		\end{align*}
		\item Set $\rho_{i+1} = 1/F(|S_{i+1}|/\delta_i)$, and $B_{i+1} = B(S_{i+1}, \rho_{i+1})$.
		\item If $\norm{I|_{P_{i+1}} - I|_{P_i}}_{L^2} > \frac{1}{F(1)}$, then increment $i$ to $i+1$, and return to step 1.
	\end{enumerate}
	Since each $P_{i+1}$ is a refinement of $P_i$, we see that $I|_{P_{i}} - I|_{P_{i - 1}}$ is constant on parts of $P_i$, whereas $I|_{P_{i+1}} - I|_{P_{i}}$ has integral 0 on such boxes. We obtain the following orthogonality:
	\begin{equation*}
	\angles{I|_{P_{i+1}} - I|_{P_{i}},I|_{P_{i}} - I|_{P_{i - 1}}} = 0,
	\end{equation*}
	and so we have the following telescoping sum:
	\begin{equation}\label{eqn:telescoping sum}
	\sum_{i = 1}^t \norm{I|_{P_{i+1}} - I|_{P_i}}_{L^2}^2 = \norm{I|_{P_{t+1}} - I|_{P_{1}}}_{L^2}^2 \le 1
	\end{equation}
	Consequently there must be some $i \le mF(1)^2$ (which in turn is bounded in terms of $\epsilon,m$) for which every $I \in \mc I$ satisfies 
	\begin{equation}\label{eqn:l2 energy gap}
	\norm{I|_{P_{i+1}} - I|_{P_i}}_{L^2} \le \frac{1}{F(1)}.
	\end{equation}
	
	Thus the procedure terminates at such a step $i$.  We now decompose each $I \in \mc I$:
	\begin{equation*}
	I = I_0+I_1+I_2,
	\end{equation*}
	where
	\begin{equation*}
	I_0 = I|_{P_i},
	\end{equation*}
	\begin{equation*}
	I_1 = I * \mu_{B_{i+1}} - I|_{P_i},
	\end{equation*}
	\begin{equation*}
	\text{and}\qquad I_2 = I - I * \mu_{B_{i+1}}.
	\end{equation*}
	
	We begin by showing that each restriction of $f_2$ to a part $p \in P_i$ has small Fourier coefficients.
	\begin{lemma} We have
		\begin{equation}
		\norm{\widehat{ I_2 \cdot \ind_{p}}}_{\ell^\infty} \lesssim \frac{1}{F(\rho_i\inv|\mc F_i|)},
		\end{equation}
		for every $I \in \mc I$ and $p \in P_i$.
	\end{lemma}
That is to say, each pointwise product $ I_2 \cdot \ind_{p}$ has Fourier coefficients that are arbitrarily small in terms of $m, |P_i|, \epsilon, \delta_i$.
\begin{proof}
	We expand 
	\begin{equation*}
	\widehat{ I_2 \cdot \ind_{p}}(\xi) = \widehat{I \cdot \ind_p}(\xi)\left(1 - \widehat\mu_{B_{i+1}}(\xi)\right).
	\end{equation*}
	Noting that $I \cdot \ind_p \in \mc F_i$ for $p \in P_i$, we see that if $\xi \notin B(S_{i+1}, \rho_{i+1})$ we necessarily have 
	\begin{equation}
	\left|\widehat{I \cdot \ind_p}(\xi)\right| \le \delta_i = 1/F( |\mc F_i|/ \delta_i).
	\end{equation} 
	We can bound $|\widehat\mu_{B_{i+1}}(\xi)|$ by 1 as $\mu_{B_{i+1}}$ is defined to have total mass 1. Consequently, we bound $|1 - \widehat\mu_{B_{i+1}}(\xi)|$ by 2, and obtain the claimed inequality in this case.
	
	Otherwise, we have $\xi \in S_{i+1}$. We begin by noting that, trivially, $\mu_{B_{i+1}}$ is supported on $B_{i+1} = B(S_{i+1}, \rho_{i+1})$. For all $x$ in this support, by definition $\xi(x) \le \rho_{i+1} \le 1/F( m|P_i|/ \epsilon\delta_i)$. Consequently we have $\exp(2\pi i \xi(x)) = 1 - O(1/F( m|P_i|/ \epsilon\delta_i))$. Since the Fourier coefficient $\widehat\mu_{B_{i+1}}(\xi)$ is an expectation of such exponentials over $x \in B_{i+1}$, it too must be $1 - O(1/F( |\mc F_i|/ \delta_i))$. Bounding the Fourier coefficient $|\widehat{I \cdot \ind_p}(\xi)|$ by 1, the claim follows in this case as well. 
\end{proof}
	\begin{lemma} We have
		%TODO for inspiration, [c.f. {[Tao]}, Lemmas 5 and 6] We have
		\begin{equation*}
		\norm{I_1}_{L^2} \lesssim \frac{1}{F(1)}.
		\end{equation*}
	\end{lemma}
	\begin{proof}
		By triangle inequality, we can write
		\begin{align*}
		\norm{I_1}_L^2 = &\norm{I|_{P_i} - I * \mu_{B_{i+1}} }_{L^2} \\
		\le
		&\norm{I|_{P_i} - (I|_{P_i}) * \mu_{B_{i+1}}}_{L^2} +\\
		&\norm{(I|_{P_i}) * \mu_{B_{i+1}} - (I|_{P_{i+1}}) * \mu_{B_{i+1}}}_{L^2}+\\
		&\norm{(I|_{P_{i+1}}) * \mu_{B_{i+1}} - I * \mu_{B_{i+1}} }_{L^2}.
		\end{align*}
		The first and third terms are bounded by Corollary \ref{cor:set convolution}; choose $\epsilon_0 = 1/F(1)$ and let $F$ grow quickly enough so that $\delta_{i+1}$ and $\rho_{i+1}$ are sufficiently small to satisfy the hypotheses of the Corollary.
		It remains to bound the second term. Applying Plancherel to (\ref{eqn:l2 energy gap}), we see
		%TODO Which Way to make the widehats look best here?
		\begin{equation*}
		\sum \left|\widehat{I|_{P_{i+1}}} - \widehat{I|_{P_{i}}}\right|^2 \le \frac{1}{F(1)^2}.
		\end{equation*}
		Since $\norm{\widehat\mu_{B_{i+1}}}_{\ell^\infty} \le 1$, we can multiply this through and obtain
		\begin{equation*}
		\sum \left|\widehat{I|_{P_{i+1}}} \widehat\mu_{B_{i+1}} - \widehat{I|_{P_{i}}} \widehat\mu_{B_{i+1}}\right|^2 \le \frac{1}{F(1)^2}.
		\end{equation*}
		Applying Plancherel again, we obtain
		\begin{equation*}
		\norm{({I|_{P_{i+1}}}) * \mu_{B_{i+1}} - ({I|_{P_{i}}}) * \mu_{B_{i+1}}}_{L^2} \le \frac{1}{F(1)}.
		\end{equation*}
		
	\end{proof}
%TODO powers of epsilon in this section.
	
	\subsection{Graph regularity}
	For this problem we will need to partition a group $G$ with respect to some functions $f : G \times G \to [0,1]$ in a way that is doubly regular. Specifically, we want a partition $\Pi$ that is graph-theoretically regular in the sense that our functions $f$ can be replaced to within a good approximation by their averages over boxes of $\Pi \times \Pi$, but we would also like the parts of $\Pi$ themselves to be pseudorandom, or Fourier uniform, as subsets of $G$. For a good reference for the various notions of graph regularity we use, see \cite{ls07}. We use the box norm, also referred to as the cut norm, which is discussed in Section 4 of \cite{ls06}. The relevant property we need is the following:
	\begin{equation}\label{eqn:box norm defn}
	\norm{F}_\square =\sup_{g,h:G \to \{0,1\}} \iint F(x,y)g(x)h(y) \lesssim \sup_{g,h:G \to [-1,1]} \iint F(x,y)g(x)h(y).
	\end{equation}
	\begin{lemma}\label{lem:full double regularity}
		Fix $t,\epsilon > 0$ and some quickly growing function $F$. Then there exist a constants $N_0$ such that the following holds. Let $G$ be a compact abelian group, and let $\mc F$ be a family of functions $f:G \to [0,1]$ with cardinality $|\mc F| \le t$. Then there exist: 
		\begin{enumerate}
			\item Three partitions $P_i$, $\Pi_i$, $\Pi$ of $G$, where $P_i = \mf B(S_i, \rho_i)$,  $|\Pi_i| =: m$, $\Pi = P_i \cap \Pi_i$, and $i,|\Pi| \le N_0$.
			\item For each $f \in \mc F$, a decomposition into $f = f_0+f_1+f_2$, such that $f_0 = f|_{\Pi \times \Pi}$, $\norm{f_1}_{L^2} \le 1/F(1/\epsilon)$, and $\norm{f_2}_\square \le 1/F(\epsilon/|\Pi|)$.
			\item For each part $\pi \in \Pi_i$, a decomposition of $I = \ind_{\pi}$ into $I_0+I_1+I_2$, such that $I_0= I|_{P_i}$, $\norm{I_1}_{L^2} \le F(\epsilon/m)$, and $\norm{\widehat{ \ind_{p} \cdot I_2}}_{\ell^\infty} \le F(\epsilon/|\Pi|)$ for every $p \in P_i$.  
		\end{enumerate}
	\end{lemma}
	\begin{proof}
		
	We create these partitions via the following iterative procedure:
	\begin{enumerate}
		\item Initialize a partition $\Pi_0 = G$, and set $i = 0$.
		\item Set $P_i$ to be the partition guaranteed by Lemma \ref{lem:strong bohr regularity}, with $\mc I := \{\ind_{\pi}\}_{\pi \in \Pi_i}.$
		\item Let $\Pi = P_i \cap \Pi_i$ be the common refinement of these partitions. Repeatedly applying weak regularity,
		create $\Pi_{i+1}$ a refinement of $\Pi$ so that $f - f|_{\Pi_i \times \Pi_i}$ has box norm less than $1/F(|\Pi|)$ for each $f \in \mc F$.
		\footnote{For reference, Lemmas 3.1 and 3.2 of \cite{ls06} do essentially this. The argument is standard: Initialize $\Pi_{i+1} = \Pi_i$. Then if there is a box $I_1 \times I_2$ on which $\int_{I_1 \times I_2} (f - f|_{\Pi_{i+1}^2}) \ge 1/F(|\Pi|)$ for some $f \in \mc F$, refine $\Pi_{i+1}$ by intersecting with $I_1$, $I_2$. A quick energy increment calculation shows that each $f$ can only force us to refine $\Pi_{i+1}$ a bounded number of times, after which the construction is complete.
		}
		\item If $\norm{f|_{ \Pi_{i+1} \times \Pi_{i+1}} - f|_{\Pi \times \Pi}}_{L^2} > 1/F(1/\epsilon)$ for any $f \in \mc F$, increment $i$ to $i+1$ and return to step (2). 
	\end{enumerate}
	Since $\Pi$ is a refinement of $P_i$, we have:
	\begin{equation*}
	\norm{f|_{ \Pi_{i+1} \times \Pi_{i+1}} - f|_{\Pi \times \Pi}}_{L^2}  \le \norm{f|_{ \Pi_{i+1} \times \Pi_{i+1}} - f|_{\Pi_i \times \Pi_i}}_{L^2}.
	\end{equation*}
	Moreover, as each $f \in F$ has bounded $L^2$ norm, by the orthogonality of these differences of projections (this is the same statement as (\ref{eqn:telescoping sum})), we may perform Step 4 only a bounded number of times in terms of $t, \epsilon, F$. After the conclusion of this procedure, for each $f \in \mc F$, the decomposition 
	\begin{equation*}
	f = f_0+f_1+f_2,
	\end{equation*} 
	satisfies the conclusions of the theorem, where
	\begin{equation*}
	f_0 = f|_{\Pi \times \Pi}
	\end{equation*}
	\begin{equation*}
	f_1 = f|_{\Pi_{i+1} \times \Pi_{i+1}} - f|_{\Pi \times \Pi}
	\end{equation*}
	\begin{equation*}
	f_2 = f - f|_{\Pi_{i+1} \times \Pi_{i+1}}.
	\end{equation*}
\end{proof}

\section{Counting}\label{sec: counting}
We now specialize to the corners problem specifically, in which we are given a subset of $G \times G$ with density $\alpha$ and want to find corners in this set. It will help to use the following symmetric formulation of this problem, in which we embed our set into the hyperplane $P = \{(x,y,z) \in G \times G \times G : x+y+z = 0\}$ by sending $(x,y) \mapsto (x,y,-x-y)$. Under this map, corners are equivalent to triples of points $(x,y,-x-y),(x,-x-z,z),(-y-z,y,z)$, and the difference $d$ equals $-x-y-z$.  Let $A \subset P$ have density $\mu(A) = \alpha$, and let $f: G \times G \to \R$ be the indicator function of the projection of $A$:
$$
f(x,y) = \ind_A(x, y, -x-y).
$$ Define $g$ and $h$ similarly for the projections onto the $(x,z)$ and $(y,z)$-planes, respectively, and apply Lemma \ref{lem:full double regularity}, regularizing with respect to the set of three functions $ \{f,g,h\} =: \mc F$.

We have now regularized our set with respect to an outer Bohr partition $P_i$, and an inner uniform partition $\Pi$. In the case of $\F_2^n$, Mandache's outer partition that is the analogue of our $P_i$ is given by the cosets of a subspace \cite{m18}. He then counts the number of corners with common difference lying in that subspace. This is convenient for him as any corner with difference lying in a subspace has all three of its points lying in a single part of $P_i^3$, and so he may restrict to individual sections of the hyperplane cut out by the boxes of $P_i^3$. This method relies on the fact that a coset of a subspace is closed under addition by elements of that subspace. Our analogy is the content of Proposition \ref{prop:sets with partitions}, in that parts of a Bohr partition are approximately closed under addition by an element of a much smaller Bohr set. Therefore, having regularized with respect to the Bohr partition $P_i = \mf B(S_i, \delta_i)$, we now count corners with difference lying in a much smaller set $B(S_i, \rho_i')$. Consequently, the vast majority of all corners we count have all three points lying in the same outer box. Here, $\rho_i'$ is an intermediate parameter that should be made sufficiently small with respect to our ``large'' parameters $\epsilon, m, |P_i|$. Anything assumed to be sufficiently small in terms of these three is also assumed to be sufficiently small in terms of $\rho_i'$. The set $B(S_i, \rho_i')$ should be thought of as lying between $B(S_i, \rho_i)$ and $P_{i+1}$, in terms of scale. Define
$$
\nu := \mu_{B(S_i, \rho_i')}.
$$

Our goal is to count count corners in $A$ with difference weighted by $\nu$. This weighted corner count is given by the integral
\begin{equation}\label{eqn:corner count}
\int f(x,y)g(x,z)h(y,z)\nu(-x-y-z).\tag{$*$}
\end{equation}

Let $B, C, D \in P_i$ and let $V = B \times C \times D$. We call such $V$ ``outer boxes.'' The partition
$\Pi$ refines each part in $P_i$ into at most $m = |\Pi_i|$ parts; say $B$ is refined into $\{B_1, \ldots, B_m\}$ and similarly for $C, D$. Then $B \times C \times D$ is refined into $m^3$ ``inner boxes'' of the form $B_i \times C_j \times D_k$. We begin by immediately applying regularity to approximate the corner count in $A$ by averages over inner boxes in $\Pi^3$.
\begin{lemma}\label{lem:counting integral}
	The corner count (\ref{eqn:corner count}) may be approximated as follows:
	\begin{align*}
	(\text{\ref{eqn:corner count}}) = O(\epsilon) +\sum_{\substack{B \times C \times D \in P_i^3\\i,j,k \in m^3}} f_0(B_i, C_j) g_0(B_i,D_k) h_0(C_j,D_k) 
	\int \ind_{B_i}(x) \ind_{C_j}(y) \ind_{D_k}(z) \nu(-x-y-z).
	\end{align*}
\end{lemma}
\begin{proof}
We break each occurrence of $f,g,h$ in (\ref{eqn:corner count}) into 3 parts by writing $f = f_0+f_1+f_2$ (similarly for $g$ and $h$). This breaks up the integral into 27 terms.

Let's look at contributions of various terms to this integral. A term that contains $f_1$ can be bounded by taking absolute values and bounding the $g,h$ terms by 1:
$$
\int f_1(x,y)g_a(x,z)h_b(y,z)\nu(-x-y-z) \le \int |f_1(x,y)|\nu(-x-y-z).
$$
Integrating over $z$ eliminates the $\nu$ term and we are left with the $L^1$-norm of $f$, which is bounded by $\epsilon$. Thus, such terms contribute $O(\epsilon)$ to the integral.

For terms that contain $f_2$, we evaluate this integral by first fixing $z$. We are using the box norm, so it will be convenient to approximate $\nu$ by a union of boxes, which is precisely the content of Corollary \ref{cor:box approx}. 

Choosing $\epsilon_0 = \epsilon$, we obtain an approximation of $\nu$ by boxes which differs from the original on a set of measure at most $\epsilon \cdot \mu(B(S_i, \rho_i'))$. Since the value of $|fgh\nu|$ is bounded by $1/\mu(B(S_i, \rho_i'))$, this part of the integral contributes at most $\epsilon$. On the remainder, we have a contribution
\begin{equation}\label{eqn:f2 contribution}
\int f_2(x,y)g_a(x,z)h_b(y,z)\cdot \frac{1}{\mu(B(S_i, \rho_i'))} ~dxdy
\end{equation}
integrated over a collection of at most $C(\epsilon, |S_i|, \rho_i')$ boxes. For fixed $z$, we can bound the integral (\ref{eqn:f2 contribution}) over any box by applying (\ref{eqn:box norm defn}). By assumption this box norm is sufficiently small in terms of $C$ and $\mu(B(S_i, \rho_i'))$ so that the sum of these integrals over all boxes in our approximation of $\nu$ can be made to be $O(\epsilon)$. Finally, integrating this $O(\epsilon)$ contribution over all $z$, we conclude that the contribution from the $f_2$ term is also $O(\epsilon)$.

Consequently, up to an $O(\epsilon)$ error, the number of corners in $A$ is given by the $f_0$, $g_0$, $h_0$ term, which is precisely the expression claimed in the lemma.
\end{proof}

It may be worthwhile to provide an outline of the rest of the proof at this point. Having now expressed the corner count in terms of a function on inner boxes, we will group these terms by their outer box. The contributions from each outer box (except a small exceptional set) can be bounded from below by $\mu(V \cap P) \cdot T(\phi_V)$, where $T$ is the functional defining Mandache's variational problem (appearing, for example, in (\ref{eqn:variational problem})), and $\phi_V$ is some function of three independent random variables that has expectation within $O(\epsilon)$ of  $\alpha(V) = \mu(V \cap A)/\mu(V \cap P)$. Consequently, the contribution from each outer box $V$ will be at least $\mu(V \cap P)m(\alpha(V)+O(\epsilon))$, which is at least $m'(\alpha)+O(\epsilon)$ by the fact that $m$ is Lipschitz \cite{m18}, the pointwise bound $m' \le m$, and the convexity of $m'$. 

We will begin the next section by defining the function $\phi$ for each $V$ and evaluating $\E[\phi]$ and $T(\phi)$, and conclude by showing that $T(\phi)$ is indeed a lower bound for the corner count derived in Lemma \ref{lem:counting integral}.

\section{Reduction to Variational Problem}
We perform the reduction described in the previous section. This follows generally the strategy in Section 3.3 of \cite{m18}, although some counts which are very easy to compute in the finite field case become more involved in the general setting (notably, Lemmas \ref{lem:T(V) in terms of avgs} and \ref{lem:evaluate weights of boxes} may each be replaced by a single line of computation or less, in the finite field setting).

Fix $V = B \times C \times D \in P_i^3$, and let $\Pi$ refine $V$ into $m^3$ inner boxes of the form $B_i \times C_j \times D_k$. Let $X$ be a random $B_i \subset B$, with weight given by
$$\P(X = B_i) = \frac{\mu(B_i)}{\mu(B)} =: \delta_{B_i}.$$
Similarly define $Y,Z$ to be random $C_j$ and $D_k$. 

We will define functions $\phi'$ and $\phi$ for each such choice of $V$. When comparing such constructions across multiple outer boxes $V$, we will use $\phi_V$ to denote the function $\phi$ constructed in box $V$.

Let $\phi': \{B_i\}\times\{C_j\}\times\{D_k\} \to \R$ be defined as follows:
 $$
\phi'(B_i,C_j,D_k) = \frac{1}{\delta_{B_i}\delta_{C_j}\delta_{D_k}} \cdot \frac{\mu(A \cap B_i \times C_j \times D_k)}{\mu(P \cap B \times C \times D)}.
$$
For now, we note:
$$
\E[\phi'] = \frac{\mu(V \cap A)}{\mu(V \cap P)}.
$$
The average of these values of $\E[\phi']$ over all boxes $V \in P_i^3$, weighted by $\mu(V \cap P)$, equals $\mu(A) = \alpha$. Indeed, as we will only ever consider the set of $V \in P_i^3$ as weighted by $\mu(V \cap P)$, we will sometimes make this implicit when referring to small fractions of the set: when we say a collection of outer boxes $X \subset P_i^3$ is at most an $\epsilon$-fraction of all outer boxes, we mean
$$
\sum_{V \in X} \mu(V \cap P) \le \epsilon.
$$
This is often quite different than the measure of $X$ as a subset of $G^3$. Similarly, when taking the expectation of some function over all outer boxes $V$, we will always do so with respect to this measure induced by the hyperplane.

%TODO give explanation here of how the proof goes, via convexity.
The desired minimization problem requires that $\phi$ has range in $[0,1]$, whereas our $\phi'$ might not; we will fix this, along with some similar normalization problems with $\phi'$, as follows. Define:
\begin{equation*}
\phi(B_i, C_j, D_k) = \begin{cases}
0 & \text{if }\min(\delta_{B_i}, \delta_{C_j}, \delta_{D_k}) < \epsilon^2/m,\\
\min(\phi', 1) & \text{otherwise.}
\end{cases}
\end{equation*}
We show this does not affect our expectation by much. To that end, we begin with a lemma:
\begin{lemma}\label{lem:phi minus phi prime}
	For all but an $O(\epsilon)$-fraction of boxes $V \in P_i^3$, we have $\E(\phi' - \phi)  = O(\epsilon)$.
\end{lemma}
\begin{proof}
	We have
	\begin{align*}
	\phi'(B_i, C_j, D_k) &= \frac{1}{\delta_{B_i}\delta_{C_j}\delta_{D_k}} \cdot \frac{\mu(A \cap B_i \times C_j \times D_k)}{\mu(P \cap B \times C \times D)}\\
	&\le \frac{1}{\delta_{B_i}\delta_{C_j}\delta_{D_k}} \cdot \frac{\mu(P \cap B_i \times C_j \times D_k)}{\mu(P \cap B \times C \times D)}.
	\end{align*}
	Let's evaluate $\mu(P \cap B_i \times C_j \times D_k)$. The set $B_i$ is the intersection of the parts $B \in P_i$ and $p_i \in \Pi_i$. For consistency of notation, write $I = \ind_{p_i}, J = \ind_{p_j}, K = \ind_{p_k}$. Consequently we can write $\ind_{B_i} = I \cdot \ind_B$, and similarly for $C_j$ and $D_k$. Thus we want to evaluate 
	\begin{equation*}
	\frac{1}{\delta_{B_i}\delta_{C_j}\delta_{D_k}\mu(V \cap P)}\int_{x,y} I\ind_B(x) J\ind_C(y) K\ind_D (-x-y).
	\end{equation*}
	Ideally, we would show that this quantity cannot be much larger than 1.

To begin, break up $I,J,K$ as described in the regularity section. We can write $I = I_0, I_1, I_2$, where these functions satisfy the conclusion of Lemma \ref{lem:strong bohr regularity}. This breaks the integral into 27 terms.

%TODO figure out what epsilons and ms are needed
We first bound terms that contain $I_1$, $J_1$, or $K_1$; without loss of generality, assume the term contains $I_1$.
Bounding $|J_a|,|K_b|$ by 1, this term contributes
\begin{equation*}
\frac{1}{\delta_{B_i}\delta_{C_j}\delta_{D_k}\mu(V \cap P)}\int_{V \cap P}|I_1|.
\end{equation*}
Now by the $L^2$ bound (which also bounds $L^1$), we have 
\begin{equation*}
\E_{V \in P_i^3}\left[ \frac{1}{\mu(V \cap P)}\int_{V \cap P}|I_1|\right] = \norm{I_1}_{L^1} \le F(\epsilon/m) \le  (\epsilon/m)^{100}
\end{equation*} 
As a consequence, in all but an $\epsilon/m$ fraction of outer boxes $V$, we have 
$$
\frac{1}{\mu(V \cap P)}\int_{V \cap P}|I_1| \le (\epsilon/m)^{99}.
$$
There are $3m$ choices of  $I, J, K$, for which an outer box may be exceptional, for a total of $O(\epsilon)$ exceptional outer boxes for this bound. In the rest, the $I_1$, $J_1$, $K_1$ terms always contribute less than
\begin{equation}\label{eqn:I1 term bound}
\frac{(\epsilon/m)^{99}}{\delta_{B_i}\delta_{C_j}\delta_{D_k}}.
\end{equation}

For terms that contain $I_2$ (or equivalently $J_2$ or $K_2$), we express our integral in terms of Fourier coefficients:
$$
\int_{x,y} I_2\ind_B(x) J_a\ind_C(y) P_b\ind_D (-x-y) = \sum_{\xi} \hat I_2 * \hat \ind_B(\xi) \hat J_a* \hat\ind_C(\xi) \hat K_b* \hat \ind_D (\xi).
$$
By our regularity assumptions we may assume the leftmost term is bounded in magnitude by some small $\epsilon_2$, so this sum is bounded by
$$
\epsilon_2 \sum_{\xi} |\hat J_a* \hat\ind_C(\xi)| |\hat K_b* \hat \ind_D (\xi)|.
$$
By Cauchy-Schwarz, this in turn is at most 
$$
\epsilon_2 \norm{\hat J_a* \hat\ind_C}_{\ell^2}\norm{\hat K_b* \hat \ind_D}_{\ell^2} = \epsilon_2 \norm{J_a\ind_C}_{L^2} \norm{K_b\ind_D}_{L^2}\le \epsilon_2.
$$
So terms of this form contribute an error on the order of $\epsilon_2/\mu(V \cap P)$, so we need to make sure $\mu(V \cap P)$ is not too small. This can be achieved easily; consider the boxes in $V \in P_i^3$ such that $\mu(V \cap P) \le \epsilon/|P_i|^3$. Summing over all such boxes, the total fraction of $P$ contained in any of these small outer boxes is at most $\epsilon$. Therefore all but an $\epsilon$-fraction of $V$ have $\mu(V \cap P) \ge \epsilon/|P_i|^3$. Returning to our computation, we can take 
$$
\norm{\widehat{I_2 \cdot \ind_B}}_{\ell^\infty} \le \epsilon_2 \le \frac{(\epsilon/m)^{100}}{|P_i|^3}.
$$
For all of the outer boxes that are not too small, we then get a contribution from $I_2$ terms of
$$
O\left(\frac{(\epsilon/m)^{99}}{\delta_{B_i}\delta_{C_j}\delta_{D_k}}\right).
$$

The only term left contains $I_0, J_0, K_0$ and is simply equal to 1, as (for example) $I_0$ is defined to be the expectation of $\ind_{p_i}$ on $B$, which is precisely $\mu(B_i)/\mu(B) = \delta_{B_i}$. Putting everything together, we have
\begin{equation}\label{eqn:phi prime close to 1}
\phi'(B_i, C_j, D_k) \le \frac{1}{\delta_{B_i}\delta_{C_j}\delta_{D_k}} \cdot \frac{\mu(P \cap B_i \times C_j \times D_k)}{\mu(P \cap B \times C \times D)} \le 1+O\left(\frac{(\epsilon/m)^{99}}{\delta_{B_i}\delta_{C_j}\delta_{D_k}}\right)
\end{equation}
Summing this up, we see
\begin{equation*}
\E[\phi' - \min(\phi', 1)] \le \sum_{i,j,k}\delta_{B_i}\delta_{C_j}\delta_{D_k} \cdot O\left(\frac{(\epsilon/m)^{99}}{\delta_{B_i}\delta_{C_j}\delta_{D_k}}\right) = O(\epsilon),
\end{equation*}
which nearly finishes the proof. We still need to show that ignoring points with $\delta_{B_i}, \delta_{C_j}$, or $\delta_{D_k}$ much smaller than average does not affect our computation by much. Let $X_i$ be the exceptional set of $p = (x,y,z) \in P$ with $\delta_{B_i(p)} \le \epsilon^2/m$, where $B_i(p)$ is the $B_i$ containing $p$. Summing this up over all $i$, we see the union of all $X_i$ has measure at most $\epsilon^2$, and performing the same process for the $y$ and $z$ coordinates gives a set of exceptional points $X$ of size $O(\epsilon^2)$. Then for all but an $\epsilon$-fraction of $V$, we have $\mu(V \cap X)/\mu(V \cap P) \le \epsilon$. In such cases, removing all points in $X$ reduces $\E[\phi']$ by an $O(\epsilon)$-fraction. As a consequence we have
$$
\E(\phi) \ge \E(\phi')(1-O(\epsilon)) - O(\epsilon).
$$
Noting that $\E(\phi) \le 1$, we see $\E(\phi') \le 1+O(\epsilon)$. Concluding, 
\begin{equation*}
\E(\phi'-\phi) = -O(\epsilon)\E(\phi') - O(\epsilon) = O(\epsilon).
\end{equation*}
\end{proof}
\begin{corollary} On all but an $\epsilon$ fraction of $V \in P_i^3$ we have:
	$$
	T(\phi') = T(\phi)+O(\epsilon).
	$$
\end{corollary}
\begin{proof}
	We expand by linearity of expectation, bounding terms like $\E[\phi'|X,Y]$ and $\E[\phi|X,Y]$ by $1+O(\epsilon)$ (this estimate follows, e.g., from (\ref{eqn:conditional exp of phi'}) below, which does not rely on this corollary).
	%TODO need to cite that this is true.
	\begin{align*}
	T(\phi') =~ &\E\big[\E(\phi'|X,Y)\E(\phi'|X,Z)\E(\phi'|Y,Z)\big]\\
	=~ &\E\big[\E(\phi|X,Y)\E(\phi|X,Z)\E(\phi|Y,Z)\big]+\\
	&\E\big[\E(\phi'-\phi|X,Y)\E(\phi|X,Z)\E(\phi|Y,Z)\big]+\\
	&\E\big[\E(\phi'|X,Y)\E(\phi'-\phi|X,Z)\E(\phi|Y,Z)\big]+\\
	&\E\big[\E(\phi'|X,Y)\E(\phi'|X,Z)\E(\phi'-\phi|Y,Z)\big]\\
	\le~&T(\phi)+(1+O(\epsilon))\Big(\E[\E(\phi'-\phi|X,Y)] +\E[ \E(\phi'-\phi|X,Z)]+\E[ \E(\phi'-\phi|Y,Z)]\Big)\\
	\le ~&T(\phi)+O(\epsilon). 
	\end{align*}
\end{proof}
As a consequence of this, we have 
$$\E_{V \in P_i^3}\big[\E[T(\phi_V)]\big] = \alpha + O(\epsilon).$$

\subsection{Computing $T(\phi)$}
We define an auxiliary function $T(V)$ as follows:
\begin{equation*}
T(V) = \sum_{i,j,k} \begin{cases}
0 & \min(\delta_{B_i},\delta_{C_j},\delta_{D_k}) < \epsilon^2/m,\\
\delta_{B_i}\delta_{C_j}\delta_{D_k}\E(\phi'|B_i,C_j)\E(\phi'|B_i,D_k)\E(\phi'|C_j,D_k)&\text{else.}
\end{cases}
\end{equation*}
Since $T(\phi) \le T(V)$, it suffices to show $T(V)$ gives us a lower bound on corner counts up to an additive error of $O(\epsilon)$.

\begin{lemma}\label{lem:box intersect P}
	On all boxes $B_i \times C_j \times D_k$ that contribute a nonzero amount to $T(V)$, we have 
	$$\frac{\mu(P \cap B_i \times C_j \times D_k)}{\delta_{B_i}\delta_{C_j}\delta_{D_k}\mu(P \cap B \times C \times D)} = 1+O(\epsilon).$$
\end{lemma}
\begin{proof}
	For contributing boxes, we have $\delta_{B_i}$, $\delta_{C_j}$, $\delta_{D_k} \ge \epsilon^2/m.$ Plug these bounds into (\ref{eqn:phi prime close to 1}). 
\end{proof}

To evaluate $T(V)$ we need to evaluate expressions of the form $\E[\phi' \mid X = B_i, Y = C_j]$. Readers familiar with Mandache's proof may recall that this was a simple computation in $\F_2^n$; that is unfortunately not the case here. We perform these calculations now.
\begin{lemma}\label{lem:T(V) in terms of avgs} For all but an $\epsilon$-fraction of $V \in P_i^3$, we have
\begin{equation*}
T(V) = O(\epsilon) + \sum_{i,j,k} \begin{cases}
0 & \min(\delta_{B_i},\delta_{C_j},\delta_{D_k}) < \epsilon^2/m,\\
\delta_{B_i}\delta_{C_j}\delta_{D_k}f_0(B_i, C_j)g_0(B_i,D_k)h_0(C_j,D_k)&\text{else.}
\end{cases}
\end{equation*}
\end{lemma}
\begin{proof}
We begin by computing
\begin{align*}
\E[\phi' \mid X = B_i, Y = C_j] &= \frac{1}{\delta_{B_i}\delta_{C_j}} \cdot \frac{\mu(A \cap B_i \times C_j \times D )}{\mu(V \cap P)}\\
&= \frac{1}{\delta_{B_i}\delta_{C_j}\mu(V \cap P)} \int I\ind_B (x) J\ind_C (y) f(x,y) \ind_D(-x-y).
\end{align*}
We also note that for terms contributing a nonzero amount to $T(V)$, we have
$$
\frac{1}{\delta_{B_i}\delta_{C_j}} \le \frac{m^2}{\epsilon^4},
$$
And that for all but an $\epsilon$-fraction of outer boxes $V$, we have
$$
\frac{1}{\mu(V \cap P)} \le \frac{|P_i|^3}{\epsilon}.
$$
We break up the contribution to our integral into various pieces.

First we write $f$ into $f_0+f_1+f_2$. For terms that contain $f_2$, we want to use the box norm bound. The function $\ind_D$ is the indicator function of a Bohr part; as such it can be broken up into boxes on which it is constant by Corollary \ref{cor:box approx}. This Corollary does not hold for the $\epsilon$-fraction of boxes with $\mu(D)$ too small, so we discard those exceptional boxes. On the rest, we can write $\ind_D$ as the union of a set with measure $\le \epsilon_0$ and a collection of $C$ boxes, where $C$ is bounded in terms of $\epsilon_0, |P_i|$. We choose $\epsilon_0$ sufficiently small in terms of $\epsilon, m, |P_i|$ so that this leftover set has measure less than $\epsilon/(\delta_{B_i}\delta_{C_j}\mu(V \cap P))$, so this part contributes at most $\epsilon$ to the integral. Since $f_2$ has sufficiently small box norm in terms of $\epsilon, m, |P_i|$, the contributions from the boxes sum to $O(\epsilon)$ as well, which finishes the bounds on the $f_2$ term.

Next we consider the $f_1$ term. Since we have a global bound on $\norm{f_1}_{L^2}$, we want to handle this term globally as well. The contribution to $\E[T(V)]$ from $f_1$ terms is bounded by:
\begin{align*}
&\le\E_{V \in P_i^3} \sum_{i,j,k} \delta_{B_i}\delta_{C_j}\delta_{D_k} \cdot \frac{1}{\delta_{B_i}\delta_{C_j}\mu(V \cap P)}\int_{P \cap B_i \times C_j \times D} |f_1|(1+O(\epsilon)) \\
&=\sum_{V \in P_i^3}\sum_{i,j}\int_{P \cap B_i \times C_j \times D}|f_1|(1+O(\epsilon))\\
&= \norm{f_1}_{L^1}(1+O(\epsilon)) \\
&= O(\epsilon^2).
\end{align*}
By Markov then, on all but an $\epsilon$-fraction of outer boxes the $f_1$ terms contribute $O(\epsilon)$ to $T(V)$. The only remaining terms contain all of $f_0$, $g_0$, $h_0$. Such a term evaluates to
\begin{equation*}
\frac{\mu(P \cap B_i \times C_j \times D)}{\delta_{B_i}\delta_{C_j}\mu(V \cap P)} f_0(B_i, C_j) = (1+O(\epsilon))f_0(B_i, C_j),
\end{equation*}
by applying Lemma \ref{lem:box intersect P}, and so we have:
\begin{equation}\label{eqn:conditional exp of phi'}
\E[\phi' \mid X = B_i, Y = C_j] = f_0(B_i, C_j) + O(\epsilon).
\end{equation}
Combining these terms gives the desired expression.
\end{proof}

We now know that $T(V)$, and consequently $\E[T(V)]$, can be approximated by a nice sum of terms involving only the averages $f_0,g_0,h_0, \delta_{B_i},\delta_{C_j},\delta_{D_k}$, and moreover $\E[T(V)]$ is within $O(\epsilon)$ of $\E[T(\phi)]$ and therefore lower bounded by the solution to Mandache's variational problem. It remains to show that this expression is a lower bound for the corner count derived in Lemma (\ref{lem:counting integral}).

\begin{lemma}\label{lem:evaluate weights of boxes}
	For all but an $\epsilon$-fraction of $V \in P_i^3$ and all $B_i \times C_j \times D_k$ contributing to $T(V)$, we have
	\begin{equation*}
	\int \ind_{B_i}(x) \ind_{C_j}(y) \ind_{D_k}(z) \nu(-x-y-z) \ge (1 + O(\epsilon)) \delta_{B_i}\delta_{C_j}\delta_{D_k} \mu(P \cap B \times C \times D).
	\end{equation*}
\end{lemma}
\begin{proof}
Expanding products of indicator functions, the left-hand side above becomes:
$$
\int I\ind_B (x) J\ind_C(y)K \ind_D (z) \nu(-x-y-z).
$$
Break up the $I = I_0+I_1+I_2$, and similarly for $J,K$. Recall that on contributing inner boxes in non-exceptional outer boxes, we have 
$$
\delta_{B_i}\delta_{C_j}\delta_{D_k}\mu(P \cap B \times C \times D) \gg_{\epsilon,m,|P_i|} 1.
$$
Assume we are dealing with a term containing $I_2$,$J_2$ or $K_2$. Then since $\norm{\widehat{I_2\ind_B}}_{\ell^\infty}$ may be assumed to be sufficiently small in terms of $\epsilon, m, |P_i|$, this contribution may immediately be bounded by $\epsilon\delta_{B_i}\delta_{C_j}\delta_{D_k} \mu(P \cap B \times C \times D)$ via Plancherel and Cauchy-Schwarz.
If we are dealing with an $I_1$ term, take absolute values and bound $|J_a|, |K_b|$ by 1, obtaining a contribution of 

$$
\int| I_1|\ind_B (x) \ind_C(y)\ind_D (z) \nu(-x-y-z)
$$ We integrate over $z$ first: Applying Corollary \ref{cor:set convolution}, we have
$$
\int_z\ind_D(z)\nu(-x-y-z) =\ind_D * \nu(-x-y)\approx \ind_{D}(-x-y).
$$ 
In particular we may replace one for the other and incur an arbitrarily small $L^2$ penalty (in terms of, say, $\epsilon, m, |P_i|$).
Making this substitution, we now want to compute
$$
\int| I_1|\ind_B (x) \ind_C(y)\ind_D (-x-y),
$$
which is $O(\epsilon/m)^{99}\mu(P \cap B \times C \times D)$ by (\ref{eqn:I1 term bound}). Applying our lower bounds on $\delta_{B_i}, \delta_{C_j},\delta_{D_k}$, this error is indeed $O(\epsilon)\cdot\delta_{B_i}\delta_{C_j}\delta_{D_k}\mu(P \cap B \times C \times D)$.
Finally for terms that are constant on $\ind_B$, that is just 
$$\delta_i \delta_j \delta_k \int \ind_B \ind_C \ind_D \nu_{B'}(-x-y-z).$$
Applying Corollary \ref{cor:set convolution} again, this is within $1+O(\epsilon)$ of 
$\int \ind_B(x)\ind_C(y) \ind_D(-x-y)$, which completes the proof.
\end{proof}
We are now in a position to prove our main theorem.
\begin{proof}[Proof of Theorem \ref{thm:main thm}]
Putting everything together, we have
\begin{align*}
&\int f(x,y)g(x,z)h(y,z) \nu (-x-y-z)\\ 
=&~O(\epsilon) +\sum_{\substack{B \times C \times D \in P_i^3\\i,j,k \in m^3}} f_0(B_i, C_j) g_0(B_i,D_k) h_0(C_j,D_k) 
\int \ind_{B_i}(x) \ind_{C_j}(y) \ind_{D_k}(z) \nu(-x-y-z). \\
\ge&~O(\epsilon)+ \sum_{\substack{B \times C \times D \in P_i^3\\i,j,k \in m^3}} f_0(B_i, C_j) g_0(B_i,D_k) h_0(C_j,D_k)\delta_i\delta_j\delta_k~ \mu(B \times C \times D \cap P)\\
\ge&~ O(\epsilon)+ \sum_{B \times C \times D \in P_i^3} \mu(B \times C \times D \cap P) T(V)\\
=&~  O(\epsilon)+ \sum_{B \times C \times D \in P_i^3} \mu(B \times C \times D \cap P) T(\phi_{B \times C \times D})\\
\ge&~ O(\epsilon)+\E_V\big[m(\alpha(V)+O(\epsilon))\big]\\
=&~O(\epsilon)+m'(\alpha),
\end{align*}
which completes the proof.
\end{proof} 

\section{Concluding remarks}
We conclude with a proof of Corollary \ref{thm:corners in Z}. To do this we simply need to include an extra character in our Bohr sets when performing the proof in $\Z/n\Z$; this strategy appears in Green's work \cite{g05}.
\begin{proof}[Proof of Corollary \ref{thm:corners in Z}]
	Embed $A \subseteq [n]^2$ in the natural way into $(\Z/n\Z)^2$ and perform the proof to count corners in $(\Z/n\Z)^2$. However, when choosing each Bohr set $S_i$, include (if it is not already present) the character $x \mapsto \exp(2\pi i x/n)$ as one of the frequencies. The rest of the proof proceeds unmodified, and one obtains the correct corner count in $(\Z/n\Z)^2$, but some corners in $(\Z/n\Z)^2$ do not pull back to corners in $[n]^2$, e.g., triples that look like $(x,y), (x+d,y), (x,y+d-n)$. Here is where the modification helps. Since we are only allowing differences in a Bohr set $B$ which contains  $x \mapsto \exp(2\pi i x/n)$ of some radius $\rho$, every $d \in B$ lies in the interval $[-\rho n, \rho n]$. Consequently the only corners in $(\Z/n\Z)^2$ with difference in $B$ which do not pull back to corners in $[n]^2$ must have either $x$ or $y$ lying in $[0,\rho n] \cup [n-\rho n, n]$. Ensuring $\rho = O(\epsilon)$, the number of such corners is $O(\epsilon n^2|B|)$, and so deleting these bad corners we are still left with $(m'(\alpha)-O(\epsilon))n^2|B|$ corners in $A$, which is sufficient for the claim.
\end{proof}

\subsection{Acknowledgments}
The author would like to thank Ashwin Sah and Mehtaab Sawhney for their encouragement, and Yufei Zhao for suggesting the problem and for helpful comments on multiple drafts of this paper. The author is supported by an NSF GRFP fellowship.

\bibliographystyle{acm}
\bibliography{Popular_Corners}

\end{document}